\documentclass{amsart}

\usepackage{titling}
\usepackage{amsmath,amssymb,url,color}
\usepackage{authblk,abstract}
\usepackage{comment}

\numberwithin{equation}{section}

\newtheorem{theorem}{Theorem}[section]
\newtheorem{lemma}[theorem]{Lemma}
\newtheorem{proposition}[theorem]{Proposition}
\newtheorem{corollary}[theorem]{Corollary}

\theoremstyle{definition}
\newtheorem{definition}[theorem]{Definition}

\newtheorem{question}[theorem]{Question}
\newtheorem{example}[theorem]{Example}

\newcommand{\meet}{\wedge}
\newcommand{\join}{\vee}

\title{Residuation algebras with functional duals
    }\thanks{The research of the second author was supported by the Vidi grant 016.138.314 of the Netherlands Organization for Scientific Research
    (NWO), by the NWO Aspasia grant 015.008.054, and by a Delft Technology Fellowship awarded in 2013. We wish to thank Peter Jipsen for his careful reading and very useful comments on an earlier draft of this paper.}
\author[1]{Wesley Fussner}
\author[2,3]{Alessandra Palmigiano}
\affil[1]{\small Department of Mathematics, University of Denver}
\affil[2]{\small Faculty of Technology, Policy and Management, Delft University of Technology}
\affil[3]{\small Department of Pure and Applied Mathematics, University of Johannesburg}

\begin{document}
\maketitle
\saythanks

\begin{abstract}
We employ the theory of canonical extensions to study residuation algebras whose associated relational structures are \emph{functional}, i.e., for which the ternary relations associated to the expanded operations admit an interpretation as (possibly partial) functions. Providing a partial answer to a question of Gehrke, we demonstrate that no universal first-order sentence in the language of residuation algebras is equivalent to the functionality of the associated relational structures.
\end{abstract}

\section{Introduction}
 In the context of  a research program aimed at establishing systematic connections between the foundations of automata theory in computer science and  duality theory in logic, in \cite{gehrke2016stone}, Gehrke specializes extended Stone and Priestley dualities in the tradition of \cite{goldblatt1989varieties} so as to capture {\em topological algebras}\footnote{For any algebraic similarity type $\tau$, a topological algebra of type $\tau$ is an algebra of type $\tau$ in the category of topological spaces, i.e.~it is a topological space endowed with continuous operations for each  $f\in \tau$.} as dual spaces. Specifically, topological algebras based on Stone  spaces are characterized as those relational Stone  spaces, as in \cite{goldblatt1989varieties}, in which the $(n+1)$-ary relations  dually corresponding to $n$-ary operations on Boolean algebras  are  {\em functional}, and an analogous result is obtained for topological algebras based on Priestley spaces. In particular, focusing the presentation on residuation algebras (see ~Definition \ref{def:residuation algebra}), the additional operations on distributive lattices are characterized for which the  dual  relations are functional (see ~\cite[Proposition 3.16]{gehrke2016stone}). These results are formulated and proved without explicit reference to the theory of canonical extensions.

This note is motivated by a question raised in \cite[end of Section 3.2]{gehrke2016stone}, viz. whether the conditions of the statement of \cite[Proposition 3.16]{gehrke2016stone} are equivalent to a first-order property of residuation algebras. To address this question, we have recast some of the notions and facts pertaining to residuation algebras in the language and theory of
canonical extensions, which allows for these facts to be reformulated  independently of specific duality-theoretic representations. Our contributions are as follows.

Firstly, we obtain a more modular and transparent understanding of how the validity of the inequality $a\backslash (b\join c) \leq (a\backslash b)\join (a\backslash c)$ forces the functionality of the dual relation.\footnote{Note that $(a\backslash b)\join (a\backslash c)\leq a\backslash (b\join c)$ holds in every residuation algebra by the monotonicity of $\backslash$ in its second coordinate, and hence $a\backslash (b\join c)  \leq  (a\backslash b)\join (a\backslash c)$ is equivalent to $a\backslash (b\join c) = (a\backslash b)\join (a\backslash c)$.}
%
In each setting (Boolean, distributive), the validity of this inequality forces the product of join-irreducible elements (which is a closed element, by the general theory of $\pi$-extensions of normal dual operators) to be either $\bot$ or finitely join prime (cf.~Proposition \ref{cor:dualrelation}). Moreover,  prime closed elements of the canonical extension of a general lattice expansion are completely join-irreducible (see~Lemma \ref{lem:primeimpji}). The functionality of the dual relation is obtained as a consequence of these two facts, of which only the first depends on the validity of the inequality above.

Secondly, we provide a partial answer to the initial question. Specifically, functionality cannot be captured by any equational condition or quasiequational condition, since there is no first-order {\em universal} sentence in the language of residuation algebras (or even residuated lattices) that is equivalent to functionality (see ~Example \ref{ex:no universal}).

Thirdly and finally, we articulate a version of \cite[Proposition 3.16]{gehrke2016stone}---reformulated in a purely algebraic fashion---in which one of the equivalent conditions in the statement is made weaker, and the corresponding part of the proof is simplified and rectified (see ~Proposition \ref{prop 3.16 emended}).

\section{Residuation algebras and their canonical extensions}
\begin{definition}(cf.~\cite{gehrke2016stone}, Definition 3.14)\label{def:residuation algebra}
A  {\em residuation algebra} is a structure $\mathbb{A} = (A, \backslash, \slash)$ such that $A$ is a bounded distributive lattice, $\backslash$ and $\slash$ are binary operations on $A$ such that $\backslash$ (resp.~$\slash$) preserves finite (hence also empty) meets in its second (resp.~first) coordinate, and  for all $a, b,c \in A$,
\[b\leq a\backslash c \quad \mbox{ iff }\quad a\leq c\slash b.\]
The {\em canonical extension} of  $\mathbb{A}$ as above is the algebra $\mathbb{A}^\delta = (A^\delta, \backslash^\pi, \slash^\pi)$ such that $A^\delta$ is the canonical extension of $A$ (see~\cite[Definition 2.5]{gehrke2001bounded}), and $\backslash^\pi$ and $\slash^\pi$ are the $\pi$-extensions of $\backslash$ and $\slash$, respectively (see~\cite[Definition 4.1]{gehrke2001bounded}).
\end{definition}
The residuation condition of the definition above implies that $\backslash$ (resp.~$\slash$) converts finite (hence empty) joins in its first (resp.~second) coordinate into meets. Together with the meet-preservation properties mentioned in the definition above, this implies (see~\cite[Lemma 4.6]{gehrke2001bounded}) that $\backslash^\pi$ and $\slash^\pi$ preserve {\em arbitrary} meets in their order-preserving coordinates and reverse {\em arbitrary} joins in their order-reversing coordinates. Since $A^\delta$ is a complete lattice, this implies that an operation $\cdot : A^\delta \times A^\delta\to A^\delta$ exists which is completely join-preserving in each coordinate and such that for all $u, v, w \in A^\delta$,
\[v\leq u\backslash^\pi w\quad \mbox{ iff }\quad u\cdot v\leq w \quad \mbox{ iff }\quad u\leq w\slash^\pi v.\]
Hence, $\mathbb{A}^\delta$ is a complete residuation algebra endowed with the structure of a complete lattice-ordered residuated groupoid.
Moreover, $\cdot$ restricts to the elements of the meet-closure\footnote{The join-closure of $A$ in $A^\delta$ is denoted $O(A^\delta)$.} of $A$ in $A^\delta$, denoted $K(A^\delta)$  (see~\cite[Lemma 10.3.1]{conradie2016algorithmic}).
\begin{definition}\label{def: a delta plus} For any residuation algebra $\mathbb{A}$ as above,
its associated  relational {\em dual structure}  $\mathbb{A}^\delta_{+} := (J^\infty(A^\delta), \geq, R)$ is based on the set $J^\infty(A^\delta)$ of the completely join-irreducible elements\footnote{$x\in A^\delta$ is {\em completely join-irreducible} if $x = \bigvee S$ implies $x\in S$ for any $S\subseteq A^\delta$. If $A$ is distributive, $A^\delta$ is completely distributive and hence completely join-irreducible elements are {\em completely join-prime}, i.e. for any $S\subseteq A^\delta$, if $x \leq \bigvee S$ then $x\leq  s$ for some $s\in S$.} of $A^\delta$ with the converse order inherited from $A^\delta$, and endowed with the ternary relation $R$ on $J^\infty(A^\delta)$ defined for $x, y, z\in J^\infty(A^\delta)$ by
\[R(x, y, z)\quad \mbox{ iff }\quad x\leq y\cdot z.\]
Such an $R$ is {\em functional} if $y\cdot z\in J^\infty(A^\delta)\cup\{\bot\}$ for all $y, z\in J^\infty(A^\delta)$, in which case we also say that $\mathbb{A}^\delta_{+}$ is {\em functional}, and is {\em functional and defined everywhere} if $y\cdot z\in J^\infty(A^\delta)$ for all $y, z\in J^\infty(A^\delta)$. In this case, we say that $\mathbb{A}^\delta_{+}$ is {\em total}.\footnote{Notice that  functional relations as defined in \cite[Definition 3.1]{gehrke2016stone} correspond to relations which are functional and defined everywhere in the present paper.}
 \end{definition} Group relation algebras, full relation algebras over a given set, and semilinear residuated lattices give examples of residuation algebras whose dual structures are functional.

 Notice that by allowing the possibility that $y\cdot z = \bot$, we are allowing  the set $R^{-1}[y, z] := \{x \mid R(x, y, z)\}$ to be empty for some $y, z\in J^\infty(A^\delta)$. 
 We emphasize that it is not uncommon that $y\cdot z=\bot$ for $y, z\in J^\infty(A^\delta)$. For instance, in any finite Boolean algebra, where $\backslash$ and $\slash$ coincide with the Boolean implication and $\cdot$ coincides with $\meet$, the product of two distinct join-irreducible elements is $\bot$. Examples of algebras in which the product of join-irreducibles may be $\bot$ are also found among MV-algebras and Sugihara monoids. 
A residuation algebra $\mathbb{A}$ as above \emph{has no zero-divisors} if $x\cdot y\neq\bot$ for all $x,y\in J^\infty(A^\delta)$.

\bigskip

The next two lemmas give a useful connection between the duality-theoretic perspective of \cite{gehrke2016stone} and the setting of canonical extensions. Specifically, they capture in a purely algebraic fashion one key property of prime filters of {\em general} lattices, namely that each prime filter induces a maximal filter/ideal pair, given by itself and its complement. This fact underlies why primeness implies join-irreducibility.
\begin{lemma}\label{lem:splitting}
For any lattice $L$, if $k\in K(L^\delta)$ is finitely prime\footnote{$u\in L^\delta$ is {\em finitely prime} if $u\neq \bot$ and for all $v, w\in L^\delta$, if $u\leq v\vee w$ then $u\leq v$ or $u\leq w$.} and  $o=\bigvee \{b\in L \mid b\not\geq k\}$, then $k\not\leq o$.
\end{lemma}
\begin{proof}
By way of contradiction, suppose that $\bigwedge \{a \in L : k\leq a\}=k\leq o$. Then by compactness, there exist finite sets $A\subseteq \{a\in L : k\leq a\}$ and $B\subseteq\{b\in L : b\not\geq k\}$ such that
$$a'=\bigwedge A \leq \bigvee B = b'$$
Then $a'\geq k$, and $b'\not\geq k$ (for if not, then by the primeness of $k$ we would have $b\geq k$ for some $b\in B$, a contradiction). But then $k\leq a'\leq b'$, so $k\leq b'$, a contradiction. This settles the lemma.
\end{proof}
\begin{lemma}\label{lem:primeimpji}
For any lattice $L$, if $k\in K(L^\delta)$ is finitely prime, then $k\in J^{\infty}(L^\delta)$.
\end{lemma}
\begin{proof}
By denseness it is enough to show that  if $k=\bigvee S$ for $S\subseteq K(L^\delta)$, then $k=s$ for some $s\in S$. Let $o=\bigvee \{a\in L  \mid a\not\geq k\}$, and, toward a contradiction, assume that $s< k$ for all $s\in S$. The assumption that $S\subseteq K(L^\delta)$ implies that for each $s\in S$,
$$s=\bigwedge \{a\in L \mid a\geq s\},$$
whence for all $s\in S$ there exists $a_s\in L$ such that $a_s\geq s$ and $a_s\not\geq k$. Hence, $a_s\leq o=\bigvee \{a\in L  \mid a\not\geq k\}$ for each $s\in S$, and so  $\bigvee \{a_s \mid s\in S\}\leq o$. Therefore,
$$o\geq\bigvee \{a_s \mid s\in S\}\geq \bigvee S = k,$$
which  contradicts Lemma \ref{lem:splitting}, proving the claim.
\end{proof}
While the lemmas above hold for {\em general} lattices, the next proposition makes use of residuation algebras being based on distributive lattices.

\begin{proposition}\label{cor:dualrelation}
 For any residuation algebra $\mathbb{A}$, if $\mathbb{A}\models a\backslash (b\join c)\leq (a\backslash b)\join (a\backslash c)$, then the dual structure $\mathbb{A}^\delta_{+}$ is functional.
\end{proposition}

\begin{proof}
The inequality $a\backslash (b\join c)\leq (a\backslash b)\join (a\backslash c)$ is Sahlqvist (see~\cite[Definition 3.5]{conradie2016algorithmic}), and hence canonical (see~\cite[Theorems 7.1 and 8.8]{conradie2016algorithmic}). That is, the assumption that $\mathbb{A}\models a\backslash (b\join c)\leq (a\backslash b)\join (a\backslash c)$ implies that $\mathbb{A}^\delta\models a\backslash (b\join c)\leq (a\backslash b)\join (a\backslash c)$. Our aim is to show that for all $x, y\in J^\infty(A^\delta)$,  if $x\cdot y\neq \bot$ then $x\cdot y\in J^\infty (A^\delta)$. From $x, y\in J^\infty(A^\delta)\subseteq K(A^\delta)$, it follows that $x\cdot y\in K(A^\delta)$ (see~discussion after Definition \ref{def:residuation algebra}). Hence, by Lemma \ref{lem:primeimpji} it is enough to show that $x\cdot y$ is finitely prime.
Suppose that $x\cdot y\leq\bigvee S$ for a finite subset $S\subseteq A^\delta$.  
By residuation, $y\leq x\backslash^\pi \bigvee S \leq \bigvee \{x\backslash^\pi s \mid s\in S\}$ (here we are using $\mathbb{A}^\delta\models a\backslash (b\join c)\leq (a\backslash b)\join (a\backslash c)$). By the primeness of $y$ (here we are using distributivity), this implies that $y\leq x\backslash s$ for some $s\in S$, i.e., $x\cdot y\leq s$ for some $s\in S$, which concludes the proof. 
\end{proof}
%
The situation in which the dual relation is functional and defined everywhere is captured by the following corollary, which is an immediate consequence of the proposition above.

\begin{corollary}
 For any residuation algebra $\mathbb{A}$, if $\mathbb{A}$ has no zero-divisors and $\mathbb{A}\models a\backslash (b\join c)\leq (a\backslash b)\join (a\backslash c)$, then  $\mathbb{A}^\delta_{+}$ is total (see~Definition \ref{def: a delta plus}).
\end{corollary}


Although the inequality $a\backslash (b\join c)\leq (a\backslash b)\join (a\backslash b)$ forces the functionality of $\mathbb{A}^\delta_+$, we observe that neither this nor any other equational condition may characterize functionality. Indeed, there is no first-order universal sentence in the language of residuation algebras that is equivalent to functionality, as the following example demonstrates.

\begin{example}\label{ex:no universal}
Consider the group $\mathbb{Z}_3$ and its complex algebra, i.e., the algebra $\mathbb{A}=(\mathcal{P}(\mathbb{Z}_3),\cap,\cup,\cdot,\backslash,\slash,\{0\})$, where for $A,B\in \mathcal{P}(\mathbb{Z})$,
$$A\cdot B = \{a+b \mid a\in A,b\in B\},$$
$$A\backslash B = \{c \mid A\cdot\{c\}\subseteq B\},$$
$$A\slash B = \{c \mid \{c\}\cdot B\subseteq A\}.$$
The algebra $\mathbb{A}$ is a finite residuation algebra (indeed, a residuated lattice), hence $\mathbb{A}^\delta = \mathbb{A}$. Moreover,  $\{n\}\cdot\{m\} = \{n+m\}$ for all $n, m\in \mathbb{Z}_3$ implies that the ternary relation $R$ on $J^{\infty}(\mathcal{P}(\mathbb{Z}_3))$ arising from $\cdot$ is functional and defined everywhere, hence $\mathbb{A}^\delta_+$ is functional, and even total. However, $\{\emptyset,\{0\},\{1,2\},\mathbb{Z}_3\}$ is the universe of a subalgebra of $\mathbb{A}$ in both the language of residuated lattices and residuation algebras in which the product of join-irreducible elements may be neither $\bot$ nor join-irreducible: for instance, $\{1,2\}\cdot \{1,2\}=\mathbb{Z}_3$ is not join-irreducible. Because the satisfaction of universal first-order sentences is inherited by subalgebras, this shows that no universal first-property in the language of residuated lattices (much less residuation algebras) may characterize the functionality of  $\mathbb{A}^\delta_+$.
\end{example}

\section{Characterizing functionality}
The following proposition emends \cite[Proposition 3.16]{gehrke2016stone}. Items (2) and (3) amount to equivalent reformulations of the corresponding items in the setting of canonical extensions. Item (1) is weaker than the corresponding item in \cite[Proposition 3.16]{gehrke2016stone}, and does not stipulate that the operation $\cdot$ gives rise to a functional relation defined everywhere (see~Definition \ref{def: a delta plus}). The proof of  (1)$\Rightarrow$(2) is essentially the same as the corresponding proof in \cite[Proposition 3.16]{gehrke2016stone}; we observe that it goes through also under this relaxed assumption. The proof of (3)$\Rightarrow$(1) is simpler than the corresponding  proof in \cite[Proposition 3.16]{gehrke2016stone}, and is where the emendation takes place.
\begin{proposition}\label{prop 3.16 emended}
 The following conditions are equivalent for any residuation algebra $\mathbb{A} = (L, \backslash, \slash)$:
\begin{enumerate}
\item  The relational structure $\mathbb{A}^\delta_+$ is functional (see~Definition \ref{def: a delta plus}).
\item $\forall a, b, c \in A, \forall x\in J^{\infty}(A^\delta)[x\leq a\Rightarrow \exists a'[a' \in A\ \& \ x\leq a' \ \& \ a\backslash (b\vee c)\leq (a'\backslash b)\vee(a'\backslash c)]$.
\item For all $x\in J^{\infty}(A^\delta)$, the map $x\backslash^\pi(\_) :O(A^\delta)\to O(A^\delta)$ is $\vee$-preserving.
\end{enumerate}
\end{proposition}
\begin{proof}
(1)$\Rightarrow$(2): Let $a, b, c \in A$, and  $x\in J^{\infty}(A^\delta)$ such that $x\leq a$. We need to find some $a'\in A$ such that $x\leq a'$ and  $a\backslash (b\vee c)\leq (a'\backslash b)\vee(a'\backslash c)$. If $y\in J^{\infty}(A^\delta)$ and $y\leq a\backslash (b\vee c)$ i.e.~$a\cdot y\leq b\vee c$, then $x\cdot y\leq b\vee c$. By assumption (1) and because in distributive lattices $x, y\in J^{\infty}(A^\delta)$ are prime,   this implies that $x\cdot y\leq  b$ or $x\cdot y\leq  c$, both in the case in which $x\cdot y = \bot$ and in case $x\cdot y \neq \bot$. This can be equivalently rewritten as $y\leq x\backslash^\pi b = \bigvee\{a\backslash b\mid a\in A\mbox{ and } x\leq a\}$ or $y\leq x\backslash^\pi c = \bigvee\{a\backslash c\mid a\in A\mbox{ and } x\leq a\}$. Since $y\in J^{\infty}(A)$, this implies that $y\leq a_y\backslash b$ or $y\leq a_y\backslash c$ for some $a_y\in A$ such that $x\leq a_y$, which implies that $y\leq (a_y\backslash b)\vee  (a_y\backslash c)$. Hence, given that  $a_y\in A$ and $x\leq a_y$ for all such $a_y$,
\[a\backslash (b\vee c) = \bigvee \{y\in J^{\infty}(A)\mid y\leq a\backslash (b\vee c)\} \leq \bigvee \{(a\backslash b)\vee  (a\backslash c)\mid a\in A\mbox{ and } x\leq a\}.\]
Hence, by compactness, and the antitonicity of $\backslash$ in the first coordinate, \[a\backslash (b\vee c)\leq \bigvee \{(a_i\backslash b)\vee  (a_i\backslash c)\mid 1\leq i\leq n\} \leq (a'\backslash b)\vee  (a'\backslash c)\] where $a': = \bigwedge_{i = 1}^n a_i\in A$ and $x\leq a'$, as required.

(2)$\Rightarrow$(3): Let $x\in J^{\infty}(A^\delta)$ and $o_1, o_2\in O(A^\delta)$. We need to prove that
\begin{equation}\label{eq: join preservation}
x\backslash^\pi(o_1\vee o_2)\leq (x\backslash^\pi o_2)\vee (x \backslash^\pi o_2).
\end{equation} By definition of $\backslash^\pi$,
\[x\backslash^\pi(o_1\vee o_2) = \bigvee \{a\backslash d\mid a, d\in A\mbox{ and } x\leq a \mbox{ and } d\leq o_1\vee o_2 \}\]
\[x\backslash^\pi o_1 = \bigvee \{a'\backslash b\mid a', b\in A\mbox{ and }  x\leq a' \mbox{ and } b\leq o_1 \}\]
\[x\backslash^\pi o_2 = \bigvee \{a'\backslash c\mid a', c\in A\mbox{ and }  x\leq a' \mbox{ and } c\leq o_2 \}\]
Thus, to prove \eqref{eq: join preservation} it is enough to show that, for all $a, d\in A$ such that $x\leq a$ and $d\leq o_1\vee o_2$,  some $a', b, c\in A$ exist such that $x\leq a'$, $b\leq o_1$, $c\leq o_2$ and  $a\backslash d\leq (a'\backslash b)\vee(a'\backslash c)$.
From $d\leq o_1\vee o_2 = \bigvee\{b\in A\mid b\leq o_1\}\vee \bigvee\{c\in A\mid c\leq o_2\}$ we get by compactness that $d\leq b\vee c$ for some $b, c\in A$ such that $ b\leq o_1$ and $c\leq o_2$. Then, by assumption (2), $a\backslash d\leq a\backslash (b\vee c)\leq (a'\backslash b)\vee(a'\backslash c)$ for some $a'\in A$ such that $x\leq a'$, as required.

(3)$\Rightarrow$(1): Let $x, y\in J^{\infty}(A^\delta)$. Then $x\cdot y\in K(A^\delta)$ because of general facts about canonical extensions of maps. Hence, by Lemma \ref{lem:primeimpji}, it is enough to show that, for all $u, v\in A^\delta$, if $x\cdot y\neq \bot$ and $x\cdot y\leq u\vee v$ then $x\cdot y\leq u$ or $x\cdot y\leq v$. By denseness, it is enough to prove the claim for $u, v\in O(A^\delta)$, and by compactness, it is enough to prove the claim for $u = b\in A$ and $v = c\in A$.
 The assumption $x\cdot y\leq b\vee c$ can be equivalently rewritten as $y\leq x\backslash^\pi (b\vee c) = (x\backslash^\pi b)\vee (x\backslash^\pi c)$, the  equality due to assumption (3). The primeness of $y$ yields  $y\leq x\backslash^\pi b$ or $y\leq x\backslash^\pi c$, i.e.~$x\cdot y\leq b$ or $x\cdot y\leq c$, as required.
\end{proof}

\section{Conclusion}
The class of residuation algebras with functional duals is not a universal class (much less a variety) according to Example \ref{ex:no universal}, but it remains open whether the property of having a functional dual may be expressed by a first-order condition in the language of residuation algebras. We pose three other questions that are implicated by the foregoing analysis. First, what is the variety generated by the class of residuation algebras with functional duals, and (in particular) do the residuation algebras with functional duals generate the variety of all residuation algebras? Second, can the treatment given in this note be extended to residuated algebraic structures with non-distributive lattice reducts? Third, given that the canonicity of Sahlqvist inequalities is key to this result, and given that the core inequality expresses the additivity of a right residual map in its order-preserving coordinates, can we extend this result to signatures of additive or multiplicative connectives on the basis of the (constructive) canonicity theory for normal and regular connectives developed in \cite{conradie2016constructive}?
We do not presently know the answer to these questions, but their resolution would deepen our understanding of functionality and promise interesting applications.

\bibliographystyle{siam}
\bibliography{bib}
\end{document}